\documentclass[12pt,twoside,reqno]{amsart}
\usepackage{amsmath}
\usepackage{amsfonts}
\usepackage{amssymb}
\usepackage{color}
\usepackage{mathrsfs}
\usepackage{cite}
\usepackage[mathlines]{lineno}
\usepackage{geometry}
\usepackage{marginnote}%\marginpar{}, \marginnote{}
\usepackage{todonotes}%\todo{}
\allowdisplaybreaks
\newtheorem{theorem}{Theorem}[section]
\newtheorem{corollary}[theorem]{Corollary}
\newtheorem{lemma}[theorem]{Lemma}
\newtheorem{proposition}[theorem]{Proposition}

\allowdisplaybreaks
\numberwithin{equation}{section}
\begin{document}
\title{Global bounded classical solutions to a parabolic-elliptic chemotaxis model with local sensing and asymptotically unbounded motility} 
%\thanks{}
\author{Jie Jiang}
\address{Innovation Academy for Precision Measurement Science and Technology, Chinese Academy of Sciences, Wuhan 430071, HuBei Province, P.R. China}
\email{jiang@apm.ac.cn, jiang@wipm.ac.cn}

\author{Philippe Lauren\c{c}ot}
\address{Laboratoire de Math\'ematiques (LAMA), UMR~5127, Universit\'e Savoie Mont Blanc, CNRS \\ F--73000 Chamb\'ery, France}
\email{philippe.laurencot@univ-smb.fr}


\date{\today}

%%%%%%%%%%%%%%%%
%%%%%%%%%%%%%%%%
\begin{abstract}
Global existence and boundedness of classical solutions are shown for a parabolic-elliptic chemotaxis system with local sensing when the motility function is assumed to be unbounded at infinity. The cornerstone of the proof is the derivation of $L^\infty$-estimates on the second component of the system and is achieved by various comparison arguments. 
\end{abstract}
%%%%%%%%%%%%%%%%
%%%%%%%%%%%%%%%%

\maketitle

%
%     HEADLINES
%
\pagestyle{myheadings}
\markboth{\sc{J.~Jiang \& Ph.~Lauren\c cot}}{\sc{Bounded solutions to a chemotaxis model with local sensing and unbounded motility}}

%%%%%%%%%%%%%%%%
%%%%%%%%%%%%%%%%
\section{Introduction}\label{sec1}
%%%%%%%%%%%%%%%%
%%%%%%%%%%%%%%%%

Let $\Omega$ be a smooth bounded domain of $\mathbb{R}^N$, $N\ge 1$, and consider the initial-boundary value problem
\begin{subequations}\label{ks}
	\begin{align}
		& \partial_t u = \Delta\big( u \gamma(v)\big), \qquad &(t,x) &\in (0,\infty)\times \Omega, \label{ks1}\\
		& 0 = \Delta v -v + u ,  \qquad &(t,x) &\in (0,\infty)\times \Omega, \label{ks2}\\
		& \nabla\big( u\gamma(v)\big)\cdot \mathbf{n} = \nabla v\cdot \mathbf{n} = 0, \qquad &(t,x) &\in (0,\infty)\times \partial\Omega, \label{ks3}\\
		& u(0) = u^{in}, \qquad &x &\in\Omega, \label{ks4}
	\end{align}
\end{subequations}
where $\mathbf{n}$ denotes the outward unit normal vector field to $\partial\Omega$.  

The system~\eqref{ks} is a classical parabolic-elliptic simplification of the well-known fully parabolic system of partial differential equations named after Keller \& Segel, which is originally proposed in their seminal work~\cite{1971KS} to model the chemotaxis phenomenon due to a local sensing mechanism:
\begin{subequations}\label{oks}
	\begin{align}
	& \partial_t u = \Delta (u \gamma(v)),  &(t,x)&\in (0,\infty)\times \Omega, \label{oks1}  \\
	&	\tau \partial_t v =\Delta v -  v  +  u,  &(t,x)&\in (0,\infty)\times \Omega, \label{oks2}\\
	& \nabla\big( u\gamma(v)\big)\cdot \mathbf{n} = \nabla v\cdot \mathbf{n} = 0, &(t,x) &\in (0,\infty)\times \partial\Omega, \label{oks3}\\
	& (u,\tau v)(0) = (u^{in},\tau v^{in}), & x &\in\Omega, \label{oks4}
	\end{align}
\end{subequations} 
Here, $u$ and $v$ represent the cell density and the chemical concentration, respectively.  The cellular motility $\gamma$ is a positive function on $(0,\infty)$ and its dependence on $v$ accounts for the influence of the chemical signal on the motion of cells  which is better seen after expanding
\begin{equation*}
	\Delta (u \gamma(v)) = \mathrm{div}\big( \gamma(v) \nabla u + u \gamma'(v) \nabla v):
\end{equation*}
the diffusion of cells in space is governed by the values of $\gamma(v)$, while the (chemotactic) bias on their motion induced by the signal is attractive when $\gamma'(v)<0$ and repulsive when $\gamma'(v)>0$.  More recently, an extended model involving a third component $n$ accounting for the nutrient level has been introduced in \cite{2011Science} and reads:
\begin{subequations}\label{cpn}
	\begin{align}
		& \partial_t u  = \Delta (u \gamma(v)) + \theta u f(n), \qquad  &(t,x)&\in (0,\infty)\times\Omega, \label{cpn1} \\
		&	\tau \partial_t v =\Delta v -  v  +  u, \qquad  &(t,x)&\in (0,\infty)\times\Omega, \label{cpn2} \\
		& \partial_t n= \Delta n - \theta u f(n), \qquad  &(t,x)&\in 
		(t,x)\in (0,\infty)\times\Omega, \label{cpn3} 
	\end{align}
\end{subequations}
also supplemented with no-flux boundary conditions and initial conditions. In~\cite{2011Science}, the motility~$\gamma$ is assumed to be a positive non-increasing function, which reflects a repressive effect of the signal on cellular motility.  Formation of spatially periodic patterns are shown by numerical simulations and experimental analysis in a growing bacteria population merely under this motility control. We note that setting $\theta=0$ in~\eqref{cpn} cancels the coupling between $u$ and $n$ and we then recover the system~\eqref{oks} for $(u,v)$.

The mathematical analysis of the Keller--Segel-type system~\eqref{oks} and its variants (such as~\eqref{ks} or~\eqref{cpn})  has attracted a lot of interest in recent years. Besides the quasilinear structure of the cell's equation~\eqref{oks1} coupling in a nonlinear way the dynamics of the cells and the signal, the cells' equation~\eqref{oks1} features a (signal-dependent) degeneracy when $\gamma$ vanishes, which is likely to occur at infinity ($\gamma(s)=s^{-k}$, $k>0$, or $\gamma(s) = e^{-\chi s}$, $\chi>0$, for instance) or at zero ($\gamma(s) = s^k$, $k>0$, for instance). In that case, equation~\eqref{oks1} is a second-order degenerate quasilinear parabolic equation which, unlike the celebrated porous medium equation $\partial_t z - \Delta z^m=0$, $m>1$, does not have a variational structure in general. Nevertheless, energy methods turn out to be adequate in some special cases which we describe now. When $\gamma\in W^{1,\infty}(0,\infty)$ ranges in a compact subinterval of $(0,\infty)$, thereby excluding the aforementioned possible degeneracy, and $\Omega$ is convex, the existence of a global bounded solution to~\eqref{oks} is shown in space dimension $N=2$ in~\cite{TW2017}, along with the existence of global weak solutions in higher space dimensions $N\ge 3$. When the motility $\gamma$ is the specific algebraically decaying function $\gamma(s)=c_0 s^{-k}$ with $c_0> 0$, it is proved in~\cite{YoonKim17} that global bounded classical solutions whatever the value of $k>0$, provided that $c_0$ is sufficiently small. The smallness condition on $c_0$ is subsequently relaxed in~\cite{AhYo2019, Wang20} when $\tau=0$ and $k<\frac{2}{N-2}$. When $\gamma(s)=1/(c+s^k)$, $c\geq0$, global weak solutions to~\eqref{oks} are constructed in~\cite{DKTY2019} when $N\in \{1,2,3\}$ and $k\in (0,k_N)$ with $k_1=\infty$, $k_2=2$, and $k_3=4/3$. When the motility function $\gamma$ is a negative exponential, $\gamma(s)=e^{-\chi s}$ with $\chi>0$,  global existence of weak solutions is established in~\cite{BLT2020} in arbitrary space dimensions $N\geq1$. Moreover, when $N=2$, there is a threshold value of $\|u^{in}\|_1$, which is exactly the same as that of the minimal Keller--Segel system according to~\cite{JW2020}, separating two different dynamics of~\eqref{oks}: global classical solutions to~\eqref{oks} exist when $\|u^{in}\|_1$ lies beyond the threshold value, while there are initial conditions with $\|u^{in}\|_1$ above the threshold value for which the corresponding classical solution to~\eqref{oks} becomes unbounded in finite or infinite time.

In a series of works \cite{Fuji2023, FuJi2020, FuJi2021a, FuJi2021b, FuSe2022a, FuSe2022b, JLZ2022, JiLa2021, Jian2022, XiaoJiang2022, DLTW2023, LyWa2022, LyWa2023, LiJi2021}, a different approach is developed and proved to be efficient in generic cases. The main idea is to exploit the specific structure of~\eqref{oks1} to derive a quasilinear degenerate parabolic equation involving a non-local term for an auxiliary unknown function. A systematic argument based on comparison techniques, monotonicity tricks, iteration procedures as well as applications of abstract semi-group theory gives rise to advanced theories on well-posedness and qualitative behavior of the solutions. 

Let us summarize the results obtained in the literature for the initial-boundary value problem~\eqref{oks} or its variant~\eqref{cpn} (also supplemented with no-flux boundary conditions) by this alternative method. When $\tau>0$, it is proved in \cite{FuJi2021a, FuSe2022a, JLZ2022} that there exists a unique global classical solution to~\eqref{oks} for a large class of non-negative initial conditions under a rather general assumption on $\gamma$:
\begin{equation}
\gamma\in C^3((0,\infty)), \quad \gamma>0 \;\;\;\text{ on }\;\; (0,\infty), \label{Aad0}
\end{equation}
\begin{equation}\label{dec}
\gamma'\leq 0 \;\;\;\text{ on }\;\; (0,\infty),
\end{equation}
and 
\begin{equation}
\limsup_{s\to\infty} \gamma(s) < 1/\tau. \label{Aad1}
\end{equation}
The monotonicity~\eqref{dec} of $\gamma$ is actually not needed, as shown recently in~\cite{XiaoJiang2022}, where the existence of global classical solutions to~\eqref{oks} is established under the only assumptions~\eqref{Aad0} and~\eqref{Aad1}. In the same vein, global weak solutions are constructed in~\cite{LiJi2021} when $\gamma$ satisfies~\eqref{Aad0} and  $\sup_{s>0}\gamma(s)<1/\tau$, and in~\cite{DLTW2023} when $\gamma$ satisfies~\eqref{Aad0} and decays at most algebraically at infinity. All these existence results share the property that no finite time blowup occurs in~\eqref{oks}, a feature which significantly differs from the minimal Keller--Segel system or the logarithmic Keller--Segel system which both involve a linear cellular diffusion and are well-known to trigger finite time singularities.

Furthermore, the boundedness of global solutions to~\eqref{oks} is shown to be closely related to the decay property of $\gamma$ at infinity. Specifically, for $\gamma$ satisfying~\eqref{Aad0}, \eqref{dec}, and vanishing at infinity 
\begin{equation}
\lim\limits_{s\rightarrow\infty}\gamma(s)=0, \label{van0}
\end{equation} 
it is proved that exponential decay is critical in space dimension $N=2$ in the following sense:  there exist unbounded solutions to~\eqref{oks} when $\gamma(s)=e^{-\chi s}$ for some $\chi>0$ and the $L^1$-norm of $u^{in}$ is sufficiently large~\cite{FuJi2021a, JW2020}, whereas all solutions to~\eqref{oks} are bounded if $\gamma$ decays slower than a negative exponential function; that is, $\gamma$ satisfies
\begin{equation*}
	\liminf\limits_{s\rightarrow\infty}e^{\chi s}\gamma(s)>0
\end{equation*}
for all $\chi>0$, typical examples including $\gamma(s)=s^{-k_1}\log^{-k_2}(1+s)$ with $(k_1,k_2)\in (0,\infty)\times [0,\infty)$, or $\gamma(s)=e^{-s^{\alpha}}$ with $0<\alpha<1$. When $N\geq3$ and $\gamma(s)\sim s^{-k}$, boundedness of global solutions holds true for $k<N/(N-2)$~\cite{JLZ2022,FuSe2022b}, which in particular improves previous results for the specific case $\tau=0$ and $\gamma(s)=s^{-k}$ by substantially enlarging the admissible range of $k$. Recently, the above boundedness results have been improved in \cite{XiaoJiang2022} by removing the monotonicity requirement~\eqref{dec} and  replacing the asymptotically vanishing  property~\eqref{van0} by the asymptotically smallness assumption~\eqref{Aad1}. Moreover, if $\gamma$ simply satisfies \eqref{Aad0} and is bounded by positive constants from above and below, so that~\eqref{oks1} is uniformly parabolic, existence of uniform-in-time bounded solution is also proved in~\cite{XiaoJiang2022} in any bounded domain $\Omega$ in arbitrary space  dimension $N\ge 1$, thus improving the corresponding one in~\cite{TW2017} obtained in two-dimensional convex domains.

As for the initial-boundary value problem~\eqref{ks}, which is recovered from~\eqref{oks} by setting formally $\tau=0$, the existence and boundedness/unboundedness results available for~\eqref{oks} are still true under the same assumptions by setting $1/0=\infty$ in~\eqref{Aad1}. Indeed, global existence of classical solutions is proved in our previous work~\cite{JiLa2021} for motility functions $\gamma$ satisfying~\eqref{Aad0} and such that $\gamma\in L^\infty(s,\infty)$ for any $s>0$. Note that, thanks to the continuity~\eqref{Aad0} of $\gamma$, the latter is obviously equivalent to
\begin{equation}
\limsup_{s\to\infty} \gamma(s) < \infty. \label{Aa1}
\end{equation} 

\medskip

According to the above discussion, the studies of~\eqref{ks} and~\eqref{oks} performed up to now require the motility~$\gamma$ to be a bounded function on $[s,\infty)$ for any $s>0$, as implied by, either \eqref{Aad1}, or the assumption that $\gamma$ ranges in a compact subinterval of $(0,\infty)$. This raises the natural question of the dynamics of~\eqref{ks} and~\eqref{oks} when $\gamma$ becomes unbounded near infinity, and the main purpose of the present contribution is a detailed analysis of the dynamics of~\eqref{ks} in that case.

We first show that the initial-boundary value problem~\eqref{ks} has a unique global bounded classical solution when $\gamma$ is, either non-decreasing, or asymptotically unbounded.

%%%%%%%%%%%%%%%%
\begin{theorem}\label{thm1}
	Assume  that $\gamma$ satisfies~\eqref{Aad0}, as well as, either $\gamma'\ge 0$ in $(0,\infty)$, or 
\begin{equation}
	\limsup\limits_{s\rightarrow\infty} \gamma(s) = \infty. \label{A0}
\end{equation} 	
	 Suppose that $u^{in}\in W_+^{1,\infty}(\Omega)$ with $m=\|u^{in}\|_1>0$, where
	\begin{equation*}
		W^{1,\infty}_+(\Omega) := \left\{ z \in W^{1,\infty}(\Omega)\ :\ z\ge 0 \;\text{ in }\; \Omega \right\}
	\end{equation*}
	and we use the short notation $\|\cdot\|_p$ for the norm $\|\cdot\|_{L^p(\Omega)}$ with $p\in[1,\infty]$.
	
	Then there is a unique non-negative global classical solution 
	\begin{equation*}
		u\in C\big([0,\infty)\times\bar{\Omega}\big) \cap C^{1,2}\big((0,\infty)\times\bar{\Omega}\big), \qquad v\in C^{1,2}\big((0,\infty)\times\bar{\Omega}\big),
	\end{equation*} 
to~\eqref{ks} which satisfies the conservation of matter
\begin{equation}
	\|u(t)\|_1 = m = \|u^{in}\|_1, \qquad t\ge 0, \label{mc}
\end{equation} 
and is uniformly bounded; that is,
	\begin{equation}
		\sup\limits_{t\in(0,\infty)}\left(\|u(t)\|_\infty+\|v(t)\|_\infty\right)<\infty. \label{ub}
	\end{equation}
\end{theorem}
%%%%%%%%%%%%%%%%

Combining the outcome of Theorem~\ref{thm1} with \cite[Theorem~1.1]{JiLa2021} provides the existence of a unique non-negative classical solution to~\eqref{ks} under the sole assumption~\eqref{Aad0} on $\gamma$, as stated now for future reference.

%%%%%%%%%%%%%%%%
\begin{corollary}\label{cor2}
	Assume  that $\gamma$ satisfies~\eqref{Aad0} and that $u^{in}\in W_+^{1,\infty}(\Omega)$ with $m=\|u^{in}\|_1>0$. Then there is a unique non-negative global classical solution to~\eqref{ks} which satisfies the conservation of matter~\eqref{mc}.
\end{corollary}
%%%%%%%%%%%%%%%%

 Another consequence of Theorem~\ref{thm1} is the boundedness of the non-negative global classical solution to~\eqref{ks} given by Corollary~\ref{cor2} in the chemorepulsive case $\gamma'\ge 0$. The dynamics is more complex in the chemoattractive case $\gamma'\le 0$ as the solution to~\eqref{ks} may blow up in infinite time  according to~\cite{FuJi2020}.

We next supplement Theorem~\ref{thm1} with explicit $L^\infty$-bounds on $u$ or $v$ when $\gamma$ is endowed with additional properties. First, 
define the (unbounded) linear operator $\mathcal{A}$ on $L^2(\Omega)$ by 
\begin{equation}
\begin{split}
	D(\mathcal{A}) & := \left\{ z\in H^2(\Omega)\ :\ \nabla z\cdot\mathbf{n} = 0 \;\text{ on }\; \partial\Omega\right\},\\
	\mathcal{A}[z] & := z - \Delta z, \qquad z\in D(\mathcal{A}).
\end{split}\label{opA}
\end{equation}
Then $v = \mathcal{A}^{-1}[u]$ according to~\eqref{ks2} and~\eqref{ks3} and $v^{in} := \mathcal{A}^{-1}[u^{in}]\in W^{3,p}(\Omega)$ for any $1\leq p<\infty$ by standard regularity theory of elliptic equations. In particular, $v^{in}$ is bounded and positive in $\bar{\Omega}$, the latter following from Lemma~\ref{lem2} below while the former is ensured by Sobolev embeddings.

Now, we show that explicit upper and lower bounds on $v$ are available under a local monotonicity assumption on $\gamma$, implying the existence of an invariant region for the second component $v$. 

%%%%%%%%%%%%%%%%
\begin{theorem}\label{thm4}
	Assume that $\gamma$ satisfies~\eqref{Aad0} and that $u^{in}\in W_+^{1,\infty}(\Omega)$ with $m=\|u^{in}\|_1>0$. Let $(u,v)$ be the corresponding non-negative classical solution to~\eqref{ks} given by Corollary~\ref{cor2}. If
	\begin{equation}
		\gamma'\ge 0 \;\;\;\text{ on }\;\; \Big[ \min_{\bar{\Omega}}\{v^{in}\} ,\max_{\bar{\Omega}}\{v^{in}\} \Big]\,, \label{mg}
\end{equation} 
	then the solution to~\eqref{ks} satisfies the uniform bound~\eqref{ub}. More precisely, 
	\begin{equation*}
		\min_{\bar{\Omega}}\{v^{in}\}  \le v(t,x) \le \max_{\bar{\Omega}}\{v^{in}\}, \qquad (t,x)\in [0,\infty)\times\bar{\Omega}.
	\end{equation*}
\end{theorem}
%%%%%%%%%%%%%%%% 

In the same vein, we have the following explicit upper bounds on $u$. 

%%%%%%%%%%%%%%%%
\begin{theorem}\label{thm3}
		Assume that $\gamma$ satisfies~\eqref{Aad0}, along with~\eqref{mg}, and that $u^{in}\in W_+^{1,\infty}(\Omega)$ with $m=\|u^{in}\|_1>0$. Let $(u,v)$ be the corresponding  non-negative global classical solution to~\eqref{ks} provided by Corollary~\ref{cor2}. Then 
		\begin{equation*}
			\|u(t)\|_\infty \le \frac{\left\| u^{in} \gamma\big(v^{in}\big) \right\|_\infty}{\gamma(\min_{\bar{\Omega}}\{v^{in}\})} \le \frac{\|u^{in}\|_\infty \gamma\big(\max_{\bar{\Omega}}\{v^{in}\}\big)}{\gamma(\min_{\bar{\Omega}}\{v^{in}\})}\,, \qquad t\ge 0\,.
		\end{equation*} 
	Assume in addition that $\gamma'' \le 0$ in $\big[ \min_{\bar{\Omega}}\{v^{in}\} ,\max_{\bar{\Omega}}\{v^{in}\} \big]$. Then $\|u(t)\|_\infty \le \| u^{in}\|_\infty$ for $t\ge 0$.
\end{theorem}
%%%%%%%%%%%%%%%% 

We finally deal with the large time behaviour of solutions to~\eqref{ks} when $\gamma$ is non-decreasing and report that there is no pattern formation in that case, as intuitively expected in the chemorepulsive regime. 

%%%%%%%%%%%%%%%%
\begin{theorem}\label{thm5}
	Assume that $\gamma$ satisfies~\eqref{Aad0} and~\eqref{mg}. Consider $u^{in}\in W_+^{1,\infty}(\Omega)$ with $m=\|u^{in}\|_1>0$ and denote the corresponding  non-negative global classical solution to~\eqref{ks} provided by Corollary~\ref{cor2} by $(u,v)$. Then
	\begin{equation*}
		\lim\limits_{t\to\infty} \left( \left\| u(t) - \frac{m}{|\Omega|} \right\|_\infty + \left\| v(t) - \frac{m}{|\Omega|} \right\|_{W^{1,\infty}} \right) = 0.
	\end{equation*}
\end{theorem}
%%%%%%%%%%%%%%%%

A similar stability result is proved in \cite[Theorem~1.3]{AhYo2019} when $\gamma(s) = s^{-k}$ for some $k\in (0,1]$. Their proof  relies on the availability of a Liapunov functional for~\eqref{ks} when $\gamma(s) = s^{-k}$ for some $k\in (0,1]$ which turns out to be also a Liapunov functional when $\gamma$ is non-decreasing. The proof of of Theorem~\ref{thm5} is thus the same as that of \cite[Theorem~1.3]{AhYo2019}, to which we refer.

\medskip

The cornerstone of the proof of our existence and boundedness results is the derivation of an upper bound for the second component $v$, which is achieved by a new delicate comparison argument developed in the current contribution. Let us illustrate the difficulty to be faced and sketch the main idea in the simple case where $\gamma$ satisfies \eqref{Aad0} and $\gamma'\ge 0$ on $(0,\infty)$. 
To begin with, we derive from~\eqref{ks1}-\eqref{ks2} the following key identity
\begin{equation}
	\partial_t v - \gamma(v) \Delta v + v \gamma(v) = \mathcal{A}^{-1}[u\gamma(v)]\,, \qquad (t,x)\in (0,\infty)\times \Omega\,, \label{fj}
\end{equation}
recalling that $\mathcal{A}$ is defined in~\eqref{opA}.  This identity is uncovered in~\cite{FuJi2020} and has since then been used efficiently to investigate the global existence and boundedness of classical solutions to~\eqref{ks} and its variants, see \cite{Fuji2023, FuJi2021a, FuJi2021b, FuSe2022a, FuSe2022b, Jian2022, JiLa2021, JLZ2022, LyWa2022, LyWa2023, DLTW2023, LiJi2021, XiaoJiang2022}, where an upper bound on~$\gamma$ is essentially used to control the non-local source term on the right-hand side of~\eqref{fj}. Indeed, supposing that $\gamma(v)$ is bounded from above on $[0,T]$ by some possibly time-dependent upper bound $\gamma^*(T)>0$, an application of the elliptic comparison principle yields an upper control of the non-local term by a term growing linearly as a function of~$v$ which reads
\begin{equation*}
	\mathcal{A}^{-1}[u\gamma(v)] \leq \gamma^*(T) \mathcal{A}^{-1}[u] = \gamma^*(T) v\,, \qquad (t,x)\in [0,T]\times\bar{\Omega}\,.
\end{equation*}
We also refer to~\cite{JiLa2021} for a more tricky derivation of  a sublinear control of this term when~$\gamma$ is non-increasing on $(0,\infty)$. Then a systematic argument based on comparison techniques and Moser-Alikakos iterations is developed in \cite{FuJi2020, FuJi2021b, XiaoJiang2022, JiLa2021} to derive estimate for $v$ in $L^\infty\big((0,T)\times\Omega\big)$.

In contrast, when $\gamma$ is non-decreasing and unbounded at infinity, the non-local term features in principle a superlinear growth with respect to~$v$. Indeed, since $\gamma(v)\leq \gamma(\|v\|_\infty)$, it follows that
\begin{equation*}
	\mathcal{A}^{-1}[u\gamma(v)] \leq \gamma(\|v\|_\infty) v\,, \qquad (t,x)\in (0,\infty)\times\Omega\,,
\end{equation*}
and thus
\begin{equation}\label{supeq}
	\partial_t v - \gamma(v) \Delta v + v \gamma(v)\leq \gamma(\|v\|_\infty)v\,,\qquad (t,x)\in(0,\infty) \times\Omega\,.
\end{equation}
The previous methods fail in this case due to the superlinear dependence of the right-hand side on $\|v\|_\infty$.

To overcome this difficulty, we develop a novel approach to derive time-independent upper bounds on $v$ solely relying on comparison techniques. Specifically, introducing the solution $V$ to the ordinary differential equation
\begin{equation}\label{supeq2}
	\begin{split}
		\frac{dV}{dt} + V\gamma(V)  & = \gamma(\|v\|_\infty) V\,, \qquad t\ge 0\,, \\
		V(0)  & = \|v^{in}\|_\infty\,,
	\end{split}
\end{equation}
recalling that $v^{in}=\mathcal{A}^{-1}[u^{in}]$, an immediate consequence of~\eqref{supeq2} is that $V$ is a supersolution to~\eqref{supeq} and we deduce from the parabolic comparison principle that $v(t,x)\leq V(t)$ for $(t,x)\in [0,\infty)\times \bar{\Omega}$. In particular, $\|v(t)\|_\infty \le V(t)$ for $t\ge 0$ and the monotonicity of~$\gamma$, the non-negativity of $V$,  and~\eqref{supeq2} entail that $dV/dt=(\gamma(\|v\|_\infty)-\gamma(V))V\le 0$. As a result, we have $\|v(t)\|_\infty\leq V(t)\leq \|v^{in}\|_\infty$ for $t\ge 0$, as stated in Theorem~\ref{thm4}. The lower bound in Theorem~\ref{thm4} is derived with a similar argument.

In the general case where $\gamma$ satisfies~\eqref{Aad0} but need not be monotone, the main idea is to split $\gamma$ as a sum of its increasing and decreasing parts. A more delicate argument combining the above technique and the monotonicity trick developed in our previous work~\cite{JiLa2021} is carried out to derive the time-independent upper bound for $v$; see Section~\ref{sec2.2}. Once the upper bound for $v$ is obtained, we accomplish the proof according to Proposition~\ref{prop2.2}.

The remainder of this paper is organized as follows. In Section~\ref{sec1a}, we recall the local existence result and establish a blowup criterion, which ensures that global existence of classical solutions to~\eqref{ks} is a direct, though far from straightforward, consequence of the boundedness of~$v$. We then derive the uniform-in-time boundedness of $v$ by the new comparison technique and hence prove Theorem~\ref{thm1} in Section~\ref{sec2}. In addition, we discuss the continuous dependence of classical solutions of~\eqref{ks} on initial conditions. We next derive explicit $L^\infty$-estimates for $v$ and $u$ in Section~\ref{sec4} and Section~\ref{sec3}, respectively, under additional conditions on $\gamma$. Finally, we study the stabilization of solutions to~\eqref{ks} and prove Theorem~\ref{thm5} in Section~\ref{sec5}.

%%%%%%%%%%%%%%%%
%%%%%%%%%%%%%%%%
\section{Preliminaries}\label{sec1a}
%%%%%%%%%%%%%%%%
%%%%%%%%%%%%%%%%

We first state the local well-posedness of~\eqref{ks}, which can be proved as in~\cite[Lemma~3.1]{AhYo2019}.

%%%%%%%%%%%%%%%%
\begin{proposition}\label{prop2.1}
	Assume that $\gamma$ satisfies~\eqref{Aad0} and that $u^{in}\in W_+^{1,\infty}(\Omega)$ with $m=\|u^{in}\|_1>0$. Then the initial-boundary value problem~\eqref{ks} has a unique non-negative classical solution 
	\begin{equation*}
	u\in C\big([0,T_{\mathrm{max}})\times\bar{\Omega}\big) \cap C^{1,2}\big((0,T_{\mathrm{max}})\times\bar{\Omega}\big), \qquad v\in C^{1,2}\big((0,T_{\mathrm{max}})\times\bar{\Omega}\big),
	\end{equation*} 
	defined on a maximal time interval $[0,T_{\mathrm{max}})$ with $T_{\mathrm{max}}\in (0,\infty]$ which satisfies
	\begin{equation}
		\|u(t)\|_1 = m =\|u^{in}\|_1, \qquad t\in [0,T_{\mathrm{max}}). \label{mcl}
	\end{equation}
In addition, if $T_{\mathrm{max}}<\infty$, then
\begin{equation}
	\lim_{t\to T_{\mathrm{max}}} \|u(t)\|_\infty = \infty. \label{gec}
\end{equation}
\end{proposition}
%%%%%%%%%%%%%%%%

We next recall the following  lemma \cite[Lemma~2.1]{FWY2015}.

%%%%%%%%%%%%%%%%
\begin{lemma}\label{lem2}
	There is $\omega_*>0$ depending only on $\Omega$ such that, for any $f\in L^1(\Omega)$ satisfying
	\begin{equation*}
		f\ge 0 \;\;\text{ a.e. in }\; \Omega \;\;\;\text{ and }\;\; \|f\|_1=m\,,
	\end{equation*}
	there holds
	\begin{equation*}
		\mathcal{A}^{-1}[f] \ge m \omega_* \;\;\;\text{ in }\;\; \Omega\,,
	\end{equation*}
recalling that the elliptic operator $\mathcal{A}$ is defined by~\eqref{opA}.
\end{lemma}
%%%%%%%%%%%%%%%%

We now fix $\gamma$ satisfying~\eqref{Aad0} and $u^{in}\in W_+^{1,\infty}(\Omega)$ with $m=\|u^{in}\|_1>0$ and denote the corresponding classical solution to~\eqref{ks} defined on $[0,T_{\text{max}})$ provided by Proposition~\ref{prop2.1} by $(u,v)$. We then infer from~\eqref{ks} that $v$ solves
\begin{subequations}\label{veq}
	\begin{align}
		\partial_t v - \gamma(v) \Delta v + v \gamma(v) & = \mathcal{A}^{-1}[u\gamma(v)], &&\qquad (t,x)\in (0,T_{\text{max}})\times \Omega, \label{veqa} \\
		\nabla v\cdot \mathbf{n} & = 0, &&\qquad (t,x)\in (0,T_{\text{max}})\times \partial\Omega, \label{veqb}\\
		v(0) & = v^{in}, &&\qquad x\in\Omega, \label{veqc}
	\end{align}
\end{subequations}
where 
\begin{equation}
	v^{in} = \mathcal{A}^{-1}\big[ u^{in} \big]\,. \label{vin}
\end{equation}
Observe that the properties of $u^{in}$, the elliptic comparison principle, Lemma~\ref{lem2}, \eqref{mcl}, and~\eqref{vin} imply that $v^{in}\in L^\infty(\Omega)$ with $\|v^{in}\|_\infty\leq \|u^{in}\|_\infty$ and
\begin{equation}
	v(t,x)\geq v_*:=m\omega_*>0, \qquad (t,x)\in[0,T_{\mathrm{max}})\times\bar{\Omega}. \label{b00}
\end{equation}

We conclude this section by showing that an $L^\infty$-estimate on $v$ on  $(0,T)$ for some $T>0$ guarantees that $T_{\mathrm{max}}\ge T$.

%%%%%%%%%%%%%%%%
\begin{proposition}\label{prop2.2}
Under the assumption of Proposition~\ref{prop2.1}, if there is $T>0$ such that
	\begin{equation}
		\mathcal{V}(T) := \sup_{[0,T]\cap [0,T_{\mathrm{max}})}\big\{ \|v(t)\|_\infty \big\} <\infty\,, \label{vvv}
	\end{equation} 
	then $T_{\mathrm{max}}\ge T$ and 
	\begin{equation}
		 \mathcal{U}(T) := \sup_{[0,T]}\big\{ \|u(t)\|_\infty \big\} < \infty\,. \label{uuu}
	\end{equation}
In addition, if~\eqref{vvv} holds true for all $T>0$ and there is $\mathcal{V}_\infty>0$ such that $\mathcal{V}(T)\le\mathcal{V}_\infty$ for all $T>0$, then there is $\mathcal{U}_\infty>0$ such that  $\mathcal{U}(T)\le \mathcal{U}_\infty$ for all $T>0$.
\end{proposition}
%%%%%%%%%%%%%%%%

\begin{proof}
 We outline the proof here since it relies on a well-established argument already described in previous works \cite{JiLa2021, JLZ2022, XiaoJiang2022}, see also \cite{FuSe2022a, FuSe2022b}.
	
First, arguing as in~\cite[Lemma~3.2]{JiLa2021}, we derive a  H\"older estimate for $v$ in $C^\alpha\big([0,T]\times\bar{\Omega}\big)$, with a possible dependence upon $T$ of both the estimate and the exponent $\alpha\in (0,1)$. Let us mention that, as already noticed in~\cite{JiLa2021}, no monotonicity property of~$\gamma$ is needed here and the proof only requires positive upper and lower bounds on~$\gamma(v)$. 
	
The second step is the derivation of estimates for $v$ in $L^\infty\big((0,T),W^{2,r}(\Omega)\big)$ for $r>1$, still possibly depending on $T$, as well as on~$r$.  The just established H\"older continuity of $v$ enables us to regard $\gamma(v)\Delta$ as a generator of a parabolic evolution operator on $L^r(\Omega)$ for $r\in (1,\infty)$. Then estimates in $L^\infty\big((0,T),W^{2,r}(\Omega)\big)$ of $v$ follows from~\eqref{veq} and applications of the abstract theory for non-autonomous parabolic equations developed in~\cite{Aman1989, Aman1990, Aman1993, Aman1995}; see \cite[Sect.~3.3 and Sect.~4.2]{JiLa2021} for detailed proofs.
	
Finally, we finish the proof of Proposition~\ref{prop2.2} by establish estimates for $u$ in $L^\infty\big((0,T),L^r(\Omega)\big)$, $r\in [1,\infty]$, by a bootstrap argument as done in~\cite{AhYo2019}.
\end{proof}

%%%%%%%%%%%%%%%%
%%%%%%%%%%%%%%%%
\section{Global existence and boundedness}\label{sec2}
%%%%%%%%%%%%%%%%
%%%%%%%%%%%%%%%%

The purpose of this section is two-fold: we show the global existence and uniform-in-time boundedness of classical solutions to~\eqref{ks}, as well as their continuous dependence on initial data.

According to Proposition~\ref{prop2.2}, it is sufficient to prove that the $L^\infty$-norm of $v$ is bounded on $[0,T_{\mathrm{max}})$ by a time-independent bound to obtain Theorem~\ref{thm1}. In order to illustrate the approach we use and avoid some technicalities, we first provide a proof under the additional assumption that $\gamma$ is non-decreasing on $(0,\infty)$, see Section~\ref{sec2.1}. The proof in the general case is then given in Section~\ref{sec2.2}.
\subsection{$L^\infty$-estimate on $v$: the non-decreasing case} \label{sec2.1}
%%%%%%%%%%%%%%%%

%%%%%%%%%%%%%%%%
\begin{lemma}\label{lem2.2}
	Assume $\gamma$ satisfy~\eqref{Aad0} and $\gamma'\ge 0$ on $(0,\infty)$. Then
	\begin{equation}
		\|v(t)\|_\infty \le \|v^{in}\|_\infty, \qquad t\in [0,T_{\mathrm{max}}). \label{b01}
	\end{equation}
\end{lemma}
%%%%%%%%%%%%%%%%

\begin{proof}
	Owing to~\eqref{ks2}, the monotonicity of $\gamma$, and the non-negativity of $u$, there holds 
	\begin{equation*}
		u\gamma(v) \le \gamma(\|v\|_\infty) \mathcal{A}[v] \;\;\text{ in }\;\; (0,T_{\text{max}})\times \Omega.
	\end{equation*} 
	Consequently, the elliptic comparison implies that 
	\begin{equation*}
		\mathcal{A}^{-1}[u \gamma(v)] \le \gamma(\|v\|_\infty) v \;\;\text{ in }\;\; (0,T_{\text{max}})\times\Omega,
	\end{equation*}
	and it readily follows from~\eqref{veq} that $v$ satisfies
	\begin{subequations}\label{b02}
		\begin{align}
			\partial_t v - \gamma(v) \Delta v + v \gamma(v) & \le  \gamma(\|v\|_\infty) v\,, &&\qquad (t,x)\in (0,T_{\text{max}})\times \Omega\,, \label{b02a} \\
			\nabla v\cdot \mathbf{n} & = 0, &&\qquad (t,x)\in (0,T_{\text{max}})\times \partial\Omega\,, \label{b02b}\\
			v(0) & = v^{in}\,, &&\qquad x\in\Omega\,. \label{b02c}
		\end{align}
	\end{subequations}
Introducing the solution $V$ to the ordinary differential equation
\begin{subequations}\label{b03}
	\begin{align}
		\frac{dV}{dt} + V\gamma(V) & = \gamma(\|v\|_\infty) V\,, \qquad t\in (0,T_{\text{max}})\,, \label{b03a} \\
		V(0) & = \|v^{in}\|_\infty\,, \label{b03b}
	\end{align}
\end{subequations}
we infer from~\eqref{b02}, \eqref{b03}, and the parabolic comparison principle that 
\begin{equation}
	v(t,x)\le V(t), \qquad (t,x)\in (0,T_{\text{max}})\times \bar{\Omega}. \label{b04}
\end{equation}
In particular, for $t\in (0,T_{\text{max}})$, there holds $\|v(t)\|_\infty\le V(t)$ and we deduce from~\eqref{b03a}, the non-negativity of $V$, and the monotonicity of $\gamma$ that
\begin{equation*}
	\frac{dV}{dt}(t) + V(t) \gamma(V(t)) \le \gamma(V(t)) V(t), \qquad t\in (0,T_{\text{max}}).
\end{equation*}
Hence, $dV/dt\le 0$ in $(0,T_{\text{max}})$, from which~\eqref{b01} follows after using~\eqref{b03b} and~\eqref{b04}.
\end{proof}

%%%%%%%%%%%%%%%%
\subsection{$L^\infty$-estimate on $v$: the general case} \label{sec2.2}
%%%%%%%%%%%%%%%%

When $\gamma$ is not monotone but becomes unbounded near infinity, the previous argument no longer works and requires to be suitably adapted. To this end, we begin with the following auxiliary result.

%%%%%%%%%%%%%%%%
\begin{lemma}\label{lem2.3}
	Under the assumptions~\eqref{Aad0} and~\eqref{A0}, there is $s_*\ge  \|v^{in}\|_\infty$ such that 
	\begin{equation}
		\gamma(s_*) = \max_{s\in [v_*,s_*]} \{\gamma(s)\}, \label{b05}
	\end{equation}
recalling that the lower bound $v_*$ on $v$ is defined in~\eqref{b00}.
\end{lemma}
%%%%%%%%%%%%%%%%

We may obviously choose $s_*=\|v^{in}\|_\infty$ when $\gamma$ is non-decreasing.

\begin{proof}
	Let $j\ge 1$ be a positive integer and set
	\begin{align*}
		M_j & := \max\limits_{s\in \big[ v_*, j\|v^{in}\|_\infty \big]}\{\gamma(s)\}, \\
		s_j & := \sup\left\{ s\in \big[ v_*, j\|v^{in}\|_\infty \big]\ :\ \gamma(s)=M_j \right\},
	\end{align*}
	so that 
	\begin{equation*}
		\gamma(s_j) = M_j = \max\limits_{s\in \big[v_*, s_j \big]}\{\gamma(s)\}.
	\end{equation*} 
Then $(M_j)_{j\ge 1}$ and $(s_j)_{j\ge 1}$ are non-decreasing sequences of positive real numbers and the unboundedness~\eqref{A0} of $\gamma$ at infinity guarantees that
\begin{equation*}
	\lim\limits_{j\to\infty} M_j = \infty \;\;\text{ and }\;\; \lim\limits_{j\to \infty} s_j = \infty.
\end{equation*}
Consequently,
	\begin{equation*}
		j_0 := \inf\{ j\ge 1\ :\ s_j\geq \|v^{in}\|_\infty\} < \infty.
	\end{equation*}
	Setting $s_* := s_{j_0}$, it is clear that $s_*\geq \|v^{in}\|_\infty$ with
	\begin{equation*}
		\gamma(s_*) = \gamma(s_{j_0}) = M_{j_{s_0}} = \max\limits_{s\in[ v_*,s_{j_0}]}\{\gamma(s)\} = \max\limits_{s\in[ v_*,s_*]}\{\gamma(s)\},
	\end{equation*}
and the proof is complete.
\end{proof}

We next define
\begin{equation}
	\gamma_i'(s):=\begin{cases}
		0, & \qquad s\in[v_*,s_*),\\
		(\gamma'(s))_+ = \max\big\{ \gamma'(s) , 0 \big\}, & \qquad s\geq s_*,
	\end{cases} \label{b06}
\end{equation}
and
\begin{equation}
	\gamma_d'(s):=\begin{cases}
		0,& \qquad s\in[v_*,s_*),\\
		-(\gamma'(s))_- = \min\big\{ \gamma'(s) , 0 \big\},& \qquad s\geq s_*,
	\end{cases} \label{b07}
\end{equation}
with $\gamma_i(s_*)=\gamma_d(s_*)=0$ and notice that
\begin{equation}
	\gamma_i\ge 0 \ge \gamma_d \;\;\text{ on }\;\; [v_*,\infty) \label{b08}
\end{equation}
and
\begin{equation}
	\gamma(s)=\gamma(s_*)+\gamma_i(s)+\gamma_d(s), \qquad s\in [s_*,\infty). \label{b18}
\end{equation}
In addition,
\begin{equation}
	\gamma_i(s) = \gamma_d(s) = 0, \qquad s\in [v_*,s_*]. \label{b09}
\end{equation}
We also set
\begin{equation}
	\Gamma_d(s):=\int_{s_*}^s\gamma_d(\sigma)d\sigma,\qquad s\in[v_*,\infty), \label{b10}
\end{equation}
and deduce from~\eqref{b08} and~\eqref{b09} that
\begin{equation}
	\begin{split}
	& \Gamma_d(s) = 0, \qquad s\in [v_*,s_*], \\
	0 \ge & \Gamma_d(s) \ge (s-s_*) \gamma_d(s) \ge s \gamma_d(s), \qquad s\ge s_*.
	\end{split} \label{b11}
\end{equation}
With this notation, we are in a position to derive an upper bound on $u\gamma(v)$. 

%%%%%%%%%%%%%%%%
\begin{lemma} \label{lem2.4}
	There holds
	\begin{equation*}
		u\gamma(v)\leq u \big[ \gamma(s_*) + \gamma_i(\|v\|_\infty) \big] + \mathcal{A}[\Gamma_d(v)] \;\;\text{ in }\;\; (0,T_{\mathrm{max}}) \times \Omega.
	\end{equation*}
\end{lemma}
%%%%%%%%%%%%%%%%

\begin{proof} Let $t\in (0,T_{\mathrm{max}})$. We first consider the case where $\|v(t)\|_\infty\ge s_*$. Then, for $x\in\Omega$, either $v(t,x)\geq s_*$ and it follows from~\eqref{ks2}, \eqref{b18}, the monotonicity of $\gamma_i$ and $\gamma_d$, and the non-negativity of $u$ that
	\begin{align*}
		u(t,x) \gamma(v(t,x)) = & u(t,x) \gamma(s_*) + u(t,x) \gamma_i(v(t,x)) + u(t,x) \gamma_d(v(t,x))\\
		\leq & u(t,x) \gamma(s_*) + u(t,x) \gamma_i(\|v(t)\|_\infty) + \gamma_d(v(t,x)) (v-\Delta v)(t,x)\\
		= & u(t,x) \gamma(s_*) + u(t,x) \gamma_i(\|v(t)\|_\infty) - \mathrm{div}(\gamma_d(v)\nabla v)(t,x) \\
		& \qquad + v(t,x) \gamma_d(v(t,x)) + \gamma'_d(v(t,x)) |\nabla v(t,x)|^2 \\
		\le & u(t,x) \gamma(s_*) + u(t,x) \gamma_i(\|v(t)\|_\infty) - \Delta\Gamma_d(v)(t,x) + v(t,x) \gamma_d(v(t,x)).
	\end{align*}
	We next use~\eqref{b11} to arrive at
	\begin{equation*}
		u(t,x) \gamma(v(t,x))\leq u(t,x) \big[ \gamma(s_*) + \gamma_i(\|v(t)\|_\infty) \big] + \mathcal{A}[\Gamma_d(v)](t,x).
	\end{equation*}
	Or $v_*\le v(t,x)<s_*$, and there is $r>0$ such $v_*\leq v(t,y)<s_*$ for all $y\in B_r(x)\subset\Omega$ by the continuity of $v$ and ~\eqref{b00}. Therefore, $\Gamma_d(v(t))\equiv 0$ in $B_r(x)$ by~\eqref{b11} and $\mathcal{A}[\Gamma_d(v)](t,x)=0$. Consequently, recalling~\eqref{b05},
	\begin{equation*}
		u(t,x) \gamma(v(t,x)) \leq u(t,x) \gamma(s_*) \leq u(t,x) \big[ \gamma(s_*) + \gamma_i(\|v(t)\|_\infty) \big] + \mathcal{A}[\Gamma_d(v)](t,x).
	\end{equation*}	
	
	We next argue as above to conclude that, when $\|v(t)\|_\infty<s_*$, the inequality
	\begin{equation*}
		u(t,x) \gamma(v(t,x)) \leq u(t,x) \gamma(s_*) \leq u(t,x) \big[ \gamma(s_*) + \gamma_i(\|v(t)\|_\infty) \big] + \mathcal{A}[\Gamma_d(v)](t,x)
	\end{equation*}	
	also holds true for all $x\in\Omega$. This completes the proof.
\end{proof}

After this preparation, we are ready to establish an $L^\infty$-estimate on $v$.

%%%%%%%%%%%%%%%
\begin{proposition} \label{prop2.5}
	Recalling that $s_*$ is defined in Lemma~\ref{lem2.3}, there holds
	\begin{equation}
		\|v(t)\|_\infty \le s_*, \qquad t\in [0,T_{\text{max}}). \label{b12}
	\end{equation}
\end{proposition}
%%%%%%%%%%%%%%%%

\begin{proof}
It readily follows from Lemma~\ref{lem2.4} and the elliptic comparison principle that
\begin{equation*}
	\mathcal{A}^{-1}[u\gamma(v)] \leq \big[ \gamma(s_*) + \gamma_i(\|v\|_\infty) \big] v + \Gamma_d(v) \;\;\text{ in }\;\; (0,T_{\text{max}})\times\Omega.
\end{equation*}
Consequently, recalling~\eqref{veq}, we realize that $v$ satisfies
\begin{subequations}\label{b13}
	\begin{align}
		\partial_t v - \gamma(v) \Delta v + v \gamma(v) & \le \big[ \gamma(s_*) + \gamma_i(\|v\|_\infty) \big] v + \Gamma_d(v), &&\qquad (t,x)\in (0,T_{\text{max}})\times \Omega, \label{b13a} \\
		\nabla v\cdot \mathbf{n} & = 0, &&\qquad (t,x)\in (0,T_{\text{max}})\times \partial\Omega, \label{b13b}\\
		v(0) & = v^{in}, &&\qquad x\in\Omega. \label{b13c}
	\end{align}
\end{subequations}

We next introduce the solution $V$ to the following ordinary differential equation
\begin{subequations}\label{b14}
	\begin{align}
		\frac{dV}{dt} + V \gamma(V) & = \big[ \gamma(s_*) + \gamma_i(\|v(t)\|_\infty) \big] V + \Gamma_d(V), \qquad t\in (0,T_{\text{max}}), \label{b14a}\\
		V(0) & =s_*, \label{b14b}
	\end{align}
\end{subequations}
noticing that the non-positivity~\eqref{b11} of $\Gamma_d$ indeed guarantees that $V$ is well-defined on $[0,T_{\text{max}})$. Since $s_*\ge \|v^{in}\|_\infty$ by Lemma~\ref{lem2.3}, we deduce from~\eqref{b13}, \eqref{b14}, and the parabolic comparison principle that 
\begin{equation}
	\|v(t)\|_\infty\leq V(t), \qquad t\in (0,T_{\text{max}}). \label{b15}
\end{equation}
Let $T\in (0,T_{\text{max}})$. On the one hand, the continuity of~$v$, the positivity of $\gamma$, and the non-positivity~\eqref{b11} of $\Gamma_d$ imply that
\begin{equation}
	V(t) \le \mathcal{V}_T := s_* \exp\left\{ T \left[ \gamma(s_*) + \sup_{\tau\in [0,T]}\{\gamma_i(\|v(\tau)\|_\infty)\} \right] \right\}, \qquad t\in [0,T]. \label{b16}
\end{equation}
On the other hand, owing to the monotonicity of $\gamma_i$, it follows from~\eqref{b14a} and~\eqref{b16} that
\begin{equation}
	\frac{dV}{dt} + V \gamma(V) \leq \big[ \gamma(s_*) + \gamma_i(V) \big] V + \Gamma_d(V), \qquad t\in (0,T_{\text{max}}). \label{b17}
\end{equation}
Now, set $G(s) := \Gamma_d(s)-s\gamma_d(s)$ for $s\ge v_*$ and notice that $G(s)=0$ for $s\in [v_*,s_*]$ by~\eqref{b09} and~\eqref{b11}. We infer from~\eqref{b18} and~\eqref{b17} that
\begin{align*}
	\frac{d(V-s_*)_+}{dt} \leq & \big[ \gamma(s_*) + \gamma_i(V) - \gamma(V) \big] V \mathrm{sign}_+(V-s_*) + \Gamma_d(V) \mathrm{sign}_+(V-s_*)\\
	= & \big[ \gamma(s_*) + \gamma_i(V) - \gamma(s_*) - \gamma_i(V) - \gamma_d(V) \big] V \mathrm{sign}_+(V-s_*) + \Gamma_d(V) \mathrm{sign}_+(V-s_*)\\
	= & \big( G(V)-G(s_*) \big) \mathrm{sign}_+(V-s_*)\\
	= & \frac{G(V)-G(s_*)}{V-s_*} (V-s_*)_+.
\end{align*}
Now,  since $G'(s)=-s\gamma_d'(s)\geq0$ for $s\ge s_*$, we deduce from~\eqref{b16} that
\begin{equation*}
	0\leq \frac{G(V)-G(s_*)}{V-s_*} (V-s_*)_+ \leq (V-s_*)_+ \sup\limits_{s\in [s_*,\mathcal{V}_T]}\{G'(s)\}, \qquad t\in [0,T],
\end{equation*}
and we conclude that $V(t)\leq s_*$ for $t\in [0,T]$. As $T$ is arbitrary in $(0,T_{\text{max}})$, we have shown~\eqref{b12} in view of \eqref{b15}.
\end{proof}

\begin{proof}[Proof of Theorem~\ref{thm1}]
	Theorem~\ref{thm1} now readily follows from Proposition~\ref{prop2.5} and Proposition \ref{prop2.2}.
\end{proof}

\begin{proof}[Proof of Corollary~\ref{cor2}]
	Either $\gamma$ satisfies~\eqref{A0} and Corollary~\ref{cor2} is an immediate consequence of Theorem~\ref{thm1}. Or $\gamma$ belongs to $L^\infty(v_*,\infty)$ and Corollary~\ref{cor2} follows at once from~\cite[Theorem~1.1]{JiLa2021}.
\end{proof}

%%%%%%%%%%%%%%%
\subsection{Continuous dependence} \label{sec2.3}
%%%%%%%%%%%%%%%%

An interesting corollary of the previous analysis is the continuous dependence of the solutions to~\eqref{ks} with respect to the initial condition $u^{in}$, a property which is likely to follow as well from the proof of \cite[Lemma~3.1]{AhYo2019}. This property is needed later in the proof of Theorem~\ref{thm4}. We recall that, according to Corollary~\ref{cor2}, the problem~\eqref{ks} is well-posed in a classical sense under the sole assumption~\eqref{Aad0} on $\gamma$.

%%%%%%%%%%%%%%%%
\begin{proposition}\label{prop2.6}
	Assume that $\gamma$ satisfies~\eqref{Aad0} and consider a sequence $(u_j^{in})_{j\ge 1}$ in $W_+^{1,\infty}(\Omega)$ and $u^{in}\in W_+^{1,\infty}(\Omega)$ such that $\|u^{in}\|_1>0$ and 
	\begin{equation}
		\lim\limits_{j\to \infty} \|u_j^{in} - u^{in}\|_{\infty} = 0. \label{c01}
	\end{equation}
	For $j\ge 1$, let $(u_j,v_j)$ be the classical solution to~\eqref{ks} with initial condition $u_j^{in}$ and denote the classical solution to~\eqref{ks} with initial condition $u^{in}$ by $(u,v)$, their existence being guaranteed by Corollary~\ref{cor2}. Then, for any $T>0$,
	\begin{equation}
		\lim\limits_{j\to\infty} \sup_{t\in [0,T]}\big\{ \|(u_j-u)(t)\|_{\infty}  + \|(v_j-v)(t)\|_{\infty} \big\} = 0. \label{c02}
	\end{equation}
\end{proposition}
%%%%%%%%%%%%%%%%

\begin{proof}
	Setting 
	\begin{equation*}
		v_j^{in} := \mathcal{A}^{-1}[u_j^{in}], \quad j\ge 1, \;\;\text{ and }\;\; v^{in} := \mathcal{A}^{-1}[u^{in}],
	\end{equation*}
	it readily follows from~\eqref{c01} that we may assume without loss of generality that there are $m_0>0$ and $M_0>0$ such that
	\begin{equation}
		\|u_j^{in}\|_1 \ge m_0 \;\;\text{ and }\;\; \|v_j^{in}\|_\infty \le M_0, \qquad j\ge 1. \label{c03}
	\end{equation}
	We then argue as in Section~\ref{sec2.2} (with $m_0$ instead of $m$ and $M_0$ instead of $\|v^{in}\|_\infty$) to establish that there is $s_0\ge M_0$ such that
	\begin{equation}
		m_0 \omega_* \le v_j(t,x) \le s_0, \qquad (t,x)\in [0,\infty) \times\bar{\Omega}, \ j\ge 1, \label{c04}
	\end{equation}
the lower bound in~\eqref{c04} being a consequence of Lemma~\ref{lem2} and~\eqref{c03}. Owing to~\eqref{c04}, we next proceed as in \cite[Proposition~4.6]{JiLa2021} to show that there is $\alpha\in (0,1/2)$ such that
\begin{equation}
	(v_j)_{j\ge 1} \;\;\text{ is bounded in }\; BUC^{2\alpha}\big( [0,\infty),C^{2\alpha}(\bar{\Omega}) \big) .\label{c05}
\end{equation}
In addition, arguing as in \cite[Proposition~4.7]{JiLa2021}, we obtain that, for any $p\in (N,\infty)$, there is $C_0(p)>0$ such that
\begin{equation}
	\|v_j(t)\|_{W^{2,p}} \le C_0(p), \qquad t\ge 0, \ j\ge 1. \label{c06}
\end{equation}
Fix $p>2N/(1-2\alpha)$. Since $C^{2\alpha}(\bar{\Omega})$ is continuously embedded in $L^p(\Omega)$ and interpolation and Sobolev inequalities guarantee that there are $C_1(p)>0$ and $C_2(p)>0$ such that
\begin{equation*}
	\|z\|_{C^{1+\alpha}} \le C_1(p) \|z\|_{W^{3/2,p}} \le C_2(p) \|z\|_{W^{2,p}}^{3/4} \|z\|_p^{1/4}, \qquad z\in W^{2,p}(\Omega),
\end{equation*}
we infer from~\eqref{c06} and the above inequality that, for $j\ge 1$ and $(t,s)\in [0,\infty)^2$, 
\begin{align*}
	\|v_j(t)-v_j(s)\|_{C^{1+\alpha}} & \le C_2(p) \left( \|v_j(t)\|_{W^{2,p}}^{3/4} + \|v_j(s)\|_{W^{2,p}}^{3/4} \right) \|v_j(t)-v_j(s)\|_p^{1/4} \\
	& \le 2 C_0(p)^{3/4} C_2(p) |\Omega|^{1/4p}\|v_j(t)-v_j(s)\|_{C^{2\alpha}}^{1/4}.
\end{align*}  
\begin{subequations}\label{reg}
Hence, owing to~\eqref{c05}, we conclude that
\begin{equation}
	(v_j)_{j\ge 1} \;\;\text{ is bounded in }\; BUC^{\alpha/2}\big( [0,\infty),C^{1+\alpha}(\bar{\Omega}) \big). \label{c07}
\end{equation}
It then follows from~\eqref{c04}, \eqref{c07}, and parabolic regularity for linear equations \cite[Chapter~V, Theorem~1.1]{LSU1968} applied to~\eqref{ks1} considered as a linear equation for $u_j$ that there is $\alpha_1\in (0,1)$ such that
\begin{equation}
	(u_j)_{j\ge 1} \;\;\text{ is bounded in }\; C^{\alpha_1/2,\alpha_1}\big( [0,T]\times \bar{\Omega} \big) \;\;\text{ for any }\;\; T>0. \label{c08}
\end{equation}
We now use elliptic regularity theory to deduce from~\eqref{ks2} and~\eqref{c08} that
\begin{equation}
	(v_j)_{j\ge 1} \;\;\text{ is bounded in }\; C^{\alpha_1/2,2+\alpha_1}\big( [0,T]\times \bar{\Omega} \big)  \;\;\text{ for any }\;\; T>0, \label{c09}
\end{equation}
which, together with parabolic regularity applied to \eqref{ks1}, see \cite[Chapter~IV]{LSU1968}, gives that
\begin{equation}
	(u_j)_{j\ge 1} \;\;\text{ is bounded in }\; C^{(2+\alpha_1)/2,2+\alpha_1}\big( (0,T)\times \bar{\Omega} \big) \;\;\text{ for any }\;\; T>0. \label{c10}
\end{equation}
\end{subequations}
We now infer from~\eqref{reg} and the Arzel\`a-Ascoli theorem that there are a subsequence of $(u_j,v_j)_{j\ge 1}$ (not relabeled) and 
\begin{equation*}
	\tilde{u}\in C\big( [0,\infty)\times\bar{\Omega} \big) \cap C^{1,2}\big( (0,\infty)\times\bar{\Omega} \big) , \quad \tilde{v} \in C\big( [0,\infty),C^1(\bar{\Omega}) \big) 
\end{equation*}
such that
\begin{equation}
	\lim_{j\to\infty} \sup_{t\in [0,T]} \left\{ \|(u_j - \tilde{u})(t)\|_\infty + \|(v_j-\tilde{v})(t)\|_{C^1} \right\} = 0 \;\;\text{ for any }\;\; T>0. \label{c11}
\end{equation}
Since $(u_j,v_j)$ solves~\eqref{ks} with initial condition $u_j^{in}$, we combine~\eqref{reg} with the convergence~\eqref{c11} to conclude that $(\tilde{u},\tilde{v})$ is a non-negative global classical solution to~\eqref{ks} with initial condition $u^{in}$; that is, $(\tilde{u},\tilde{v}) = (u,v)$ by the uniqueness of classical solutions to~\eqref{ks}, see Proposition~\ref{prop2.1}, so that it is the whole sequence $\big( (u_j,v_j) \big)_{j\ge 1}$ which actually converges to $(u,v)$ and satisfies~\eqref{c11}. We have thus completed the proof.
\end{proof}

%%%%%%%%%%%%%%%%
%%%%%%%%%%%%%%%%
\section{Explicit upper and lower bounds on $v$}\label{sec4}
%%%%%%%%%%%%%%%%
%%%%%%%%%%%%%%%%

The main building block of the proof of Theorem~\ref{thm4} is the following lemma, which is somewhat a generalisation of Lemma~\ref{lem2.2}. 

%%%%%%%%%%%%%%%%
\begin{lemma} \label{lem4.1}
	Assume that $\gamma$ satisfies~\eqref{Aad0} and let $u^{in}\in W_+^{1,\infty}(\Omega)$ with $m=\|u^{in}\|_1>0$. We denote the corresponding global classical solution to~\eqref{ks} by $(u,v)$, see Corollary~\ref{cor2}. Assume further that there are $0<a<b$ such that
	\begin{align}
		& \gamma'\ge 0 \;\;\text{ on }\;\; [a,b], \label{d01} \\
		& a < \min_{\bar{\Omega}}\{ v^{in} \} \le \max_{\bar{\Omega}}\{ v^{in} \} < b, \label{d02}
	\end{align}
where $v^{in} = \mathcal{A}^{-1}[u^{in}]$. Then
\begin{equation*}
	a \le v(t,x) \le b, \qquad (t,x)\in [0,\infty)\times\bar{\Omega}.
\end{equation*}
\end{lemma}
%%%%%%%%%%%%%%%%

\begin{proof}
	For $t\ge 0$, we put
	\begin{equation*}
		 v_* \le \mu(t) := \min_{x\in\bar{\Omega}}\{ v(t,x)\} \le M(t) :=  \max_{x\in\bar{\Omega}}\{ v(t,x)\},
	\end{equation*}
which are well-defined and positive according to elliptic regularity and Lemma~\ref{lem2}. Owing to the continuity of $v$, it follows from~\eqref{d02} that
\begin{equation*}
	\tau := \inf\left\{ t>0\ :\ M(t)>b \;\;\text{ or }\; \mu(t)<a \right\} \in (0,\infty].
\end{equation*}
In particular,
\begin{equation}
	a \le \mu(t) \le M(t) \le b, \qquad t\in [0,\tau). \label{d07}
\end{equation}
By~\eqref{d01}, \eqref{d07}, and the non-negativity of $u$,
\begin{equation*}
	u \gamma(\mu) \le u \gamma(v) \le u\gamma(M) \;\;\text{ in }\;\; [0,\tau)\times \bar{\Omega},
\end{equation*}
and the elliptic comparison principle and~\eqref{ks2} imply that
\begin{equation}
	\gamma(\mu) v = \mathcal{A}^{-1}[\gamma(\mu)u] \le \mathcal{A}^{-1}[\gamma(v)u] \le \mathcal{A}^{-1}[\gamma(M)u] = \gamma(M) v \;\;\text{ in }\;\; [0,\tau)\times \bar{\Omega}. \label{d03}
\end{equation} 
It follows from~\eqref{veq} and~\eqref{d03} that $v$ satisfies
\begin{subequations}\label{d04}
	\begin{align}
		\gamma(M) v \ge\partial_t v - \gamma(v) \Delta v + v \gamma(v) & \ge \gamma(\mu) v, \qquad (t,x)\in (0,\tau)\times \Omega, \label{d04a} \\
		\nabla v\cdot \mathbf{n} & = 0, \qquad (t,x)\in (0,\tau)\times \partial\Omega, \label{d04b}\\
		v(0) & = v^{in}, \qquad x\in\Omega. \label{d04c}
	\end{align}
\end{subequations}
The parabolic comparison principle then entails that
\begin{equation}
	V_u(t) \ge v(t,x) \ge V_l(t), \qquad (t,x)\in [0,\tau)\times \bar{\Omega}, \label{d06}
\end{equation} 
where $V_l$ and $V_u$ are the solutions to the ordinary differential equations
\begin{subequations}\label{d05}
	\begin{align}
		\frac{dV_l}{dt} + V_l \gamma(V_l) & = \gamma(\mu) V_l, \qquad t\in (0,\tau), \label{d05a} \\
		V_l(0) & = \mu(0), \label{d05b}
	\end{align}
\end{subequations}
and
\begin{subequations}\label{d10}
	\begin{align}
		\frac{dV_u}{dt} + V_u \gamma(V_u) & = \gamma(M) V_u, \qquad t\in (0,\tau), \label{d10a} \\
		V_u(0) & = M(0), \label{d10b}
	\end{align}
\end{subequations}
respectively.
Since $b> V_u(0) \ge V_l(0)>a$ by~\eqref{d02}, 
\begin{equation*}
	\tau_* := \inf\left\{ t\in (0,\tau)\ :\ V_u(t)>b \;\;\text{ or }\;V_l(t)<a \right\} \in (0,\tau],
\end{equation*}
and we infer from~\eqref{d01}, \eqref{d07}, and~\eqref{d06} that 
\begin{equation}
	\gamma(V_l(t)) \le \gamma(\mu(t)) \le \gamma(v(t,x)) \le \gamma(M(t)) \le \gamma(V_u(t)), \qquad (t,x)\in [0,\tau_*)\times\bar{\Omega}. \label{d08}
\end{equation}
Combining~\eqref{d06}, \eqref{d05}, \eqref{d10}, and~\eqref{d08}, we deduce that
\begin{equation*}
	\frac{dV_l}{dt}(t) \ge 0 \ge \frac{dV_u}{dt}(t), \qquad t\in (0,\tau_*).  
\end{equation*}
Hence, owing to~\eqref{d02}
\begin{equation*}
	V_l(t) \ge V_l(0) > a \;\;\text{ and }\;\; V_u(t) \le V_u(0) < b, \qquad t\in [0,\tau_*),
\end{equation*}
so that $\tau_*=\tau$. Gathering~\eqref{d06} and the above inequalities, we conclude that
\begin{equation*}
	a < V_l(t) \le v(t,x) \le V_u(t) < b, \qquad (t,x)\in [0,\tau)\times\bar{\Omega}.
\end{equation*}
Consequently, $\tau=\infty$ and the proof of Lemma~\ref{lem4.1} is complete. 
\end{proof}

\begin{proof}[Proof of Theorem~\ref{thm4}]
	If $u^{in}\equiv m/|\Omega|$, then $(u,v)=(m/|\Omega|,m/|\Omega|)$ and Theorem~\ref{thm4} is obvious. Let us thus assume that $u^{in}$ is not constant, so that $v^{in}$ is also not constant. We then pick $\nu\in \big( \min_{\bar{\Omega}}\{v^{in}\} , \max_{\bar{\Omega}}\{v^{in}\}\big)$ and define 
	\begin{equation*}
		u_j^{in} := \frac{j-1}{j} u^{in} + \frac{1}{j}\nu, \qquad v_j^{in} := \mathcal{A}^{-1}[u_j^{in}] = \frac{j-1}{j} v^{in} + \frac{1}{j}\nu, \qquad j\ge 1.
	\end{equation*}
	Clearly, for each $j\ge 1$,
	\begin{align*}
		& \min_{\bar{\Omega}}\{v^{in}\} < \frac{j-1}{j} \min_{\bar{\Omega}}\{v^{in}\} + \frac{1}{j} \nu = \min_{\bar{\Omega}}\{v_j^{in}\} , \\
		& \max_{\bar{\Omega}}\{v^{in}\} > \frac{j-1}{j} \max_{\bar{\Omega}}\{v^{in}\} + \frac{1}{j} \nu = \max_{\bar{\Omega}}\{v_j^{in}\},
	\end{align*}
and we infer from Lemma~\ref{lem4.1} that the global classical solution $(u_j,v_j)$ to~\eqref{ks} with initial condition $u_j^{in}$ satisfies
\begin{equation*}
	\min_{\bar{\Omega}}\{v^{in}\} \le v_j(t,x) \le \max_{\bar{\Omega}}\{v^{in}\}, \qquad (t,x)\in [0,\infty)\times\bar{\Omega}.
\end{equation*}
We then pass to the limit $j\to\infty$ in the above inequality with the help of Proposition~\ref{prop2.6} to complete the proof. 
\end{proof}

%%%%%%%%%%%%%%%%
%%%%%%%%%%%%%%%%
\section{Explicit $L^\infty$-estimates on $u$}\label{sec3}
%%%%%%%%%%%%%%%%
%%%%%%%%%%%%%%%%

This section is devoted to the proof of Theorem~\ref{thm3} and we begin with the case of non-increasing and concave motility function $\gamma$. In this case, the evolution equation~\eqref{veqa} for $v$ is not used and the proof relies on direct computations involving~\eqref{ks1} and~\eqref{ks2} and comparison arguments.

\begin{proof}[Proof of Theorem~\ref{thm3}~(Part 1)]
	Let us thus consider a function $\gamma$ satisfying~\eqref{Aad0} and being non-decreasing and concave on $\big[ \min_{\bar{\Omega}}\{v^{in}\} ,\max_{\bar{\Omega}}\{v^{in}\} \big]$. By~\eqref{ks1} and~\eqref{ks2},
	\begin{align*}
		\partial_t u & = \mathrm{div}\big( \gamma(v) \nabla u + u \gamma'(v) \nabla v \big) \\
		& = \gamma(v) \Delta u + 2 \gamma'(v) \nabla v\cdot \nabla u + u \gamma''(v) |\nabla v|^2 + u \gamma'(v) \Delta v \\
		& = \gamma(v) \Delta u + 2 \gamma'(v) \nabla v\cdot \nabla u + u \gamma''(v) |\nabla v|^2 + u \gamma'(v) (v-u).
	\end{align*}
	Since $v=\mathcal{A}^{-1}[u]$ and $u\le \|u\|_\infty$ in $\Omega$, the elliptic comparison principle implies that
	\begin{equation*}
		v(t,x)\le \|u(t)\|_\infty, \qquad (t,x)\in [0,\infty)\times\bar{\Omega}.
	\end{equation*}
	Consequently, recalling that $v$ ranges in $\big[ \min_{\bar{\Omega}}\{v^{in}\} ,\max_{\bar{\Omega}}\{v^{in}\} \big]$ by Theorem~\ref{thm4}, it follows from the monotonicity and concavity of $\gamma$ and the non-negativity of $u$ that
	\begin{equation*}
		u \gamma''(v) |\nabla v|^2 + u \gamma'(v) (v-u) \le u \gamma'(v) \big( \|u\|_\infty - u \big) \le Y u \big( \|u\|_\infty - u \big)
	\end{equation*}
	with $Y := \gamma'\left( \min_{\bar{\Omega}}\{v^{in}\} \right) \ge 0$ and we conclude that
	\begin{subequations}\label{c1}
		\begin{align}
			& \partial_t u - \gamma(v) \Delta u -2 \gamma'(v) \nabla v\cdot \nabla u \le  Y u \big( \|u\|_\infty - u \big) \;\;\text{ in }\; (0,\infty)\times\Omega, \label{c1a} \\
			& \nabla u\cdot \mathbf{n} = 0 \;\;\text{ on }\; (0,\infty)\times\partial\Omega, \label{c1b} \\
			& u(0) = u^{in} \;\;\text{ in }\; \Omega. \label{c1c}
		\end{align}
	\end{subequations}
	Introducing the solution $U\in C^1([0,\infty))$ to the ordinary differential equation
	\begin{subequations}\label{c2}
		\begin{align}
			&\frac{dU}{dt} = Y U \big( \|u\|_\infty - U \big), \qquad t>0, \label{c2a} \\
			& U(0) = \|u^{in}\|_\infty, \label{c2b}
		\end{align}
	\end{subequations}
	it readily follows from~\eqref{c1}, \eqref{c2}, and the parabolic comparison principle that 
	\begin{equation}
		u(t,x) \le U(t), \qquad (t,x)\in [0,\infty)\times\bar{\Omega}. \label{c3}
	\end{equation}
	In particular, $\|u(t)\|_\infty \le U(t)$ for all $t\ge 0$ by~\eqref{c3} and we infer from~\eqref{c2a} and the non-negativity of $Y$ that $dU/dt\le 0$ in $(0,\infty)$. Therefore, $U(t)\le U(0)$ for $t\ge 0$ which, together with~\eqref{c2b} and~\eqref{c3}, completes the proof.
\end{proof}

We next turn to general non-decreasing motility functions $\gamma$ and use the comparison principle applied to the parabolic equation satisfied by the auxiliary function $u\gamma(v)$.

\begin{proof}[Proof of Theorem~\ref{thm3}~(Part 2)]
	We assume here that $\gamma$ satisfies~\eqref{Aad0} and~\eqref{mg}. Introducing $\varphi := u\gamma(v)$, we infer from~\eqref{ks1} and~\eqref{veqa} that 
	\begin{align*}
		\partial_t \varphi & = \gamma(v) \partial_t u + u \gamma'(v) \partial_t v = \gamma(v) \Delta\varphi + u \gamma'(v) \left( \mathcal{A}^{-1}[\varphi] - \varphi \right) \\
		& = \gamma(v) \Delta\varphi + \frac{\gamma'}{\gamma}(v) \varphi \left( \mathcal{A}^{-1}[\varphi] - \varphi \right).
	\end{align*}
	Since 
	\begin{equation*}
		\mathcal{A}^{-1}[\varphi] \le \|\varphi\|_\infty \;\;\text{ in }\; [0,\infty)\times\bar{\Omega}
	\end{equation*}
	by the elliptic comparison principle, it follows from the monotonicity of $\gamma$, the non-negativity of $\varphi$, and Theorem~\ref{thm4} that
	\begin{equation*}
		\frac{\gamma'}{\gamma}(v) \varphi \left( \mathcal{A}^{-1}[\varphi] - \varphi \right) \le \frac{\gamma'}{\gamma}(v) \varphi \left( \|\varphi\|_\infty - \varphi \right) \le Y \|\gamma'(v)\|_\infty \varphi \big( \|\varphi\|_\infty - \varphi \big),
	\end{equation*}
	with $1/Y:= \gamma(\min_{\bar{\Omega}}\{v^{in}\})>0$. Therefore, $\varphi$ satisfies
	\begin{subequations}\label{d1}
		\begin{align}
			& \partial_t \varphi - \gamma(v) \Delta\varphi \le Y \|\gamma'(v)\|_\infty \varphi \big( \|\varphi\|_\infty - \varphi \big) \;\;\text{ in }\; (0,\infty)\times\Omega, \label{d1a} \\
			& \nabla \varphi\cdot \mathbf{n} = 0 \;\;\text{ on }\; (0,\infty)\times\partial\Omega, \label{d1b} \\
			& \varphi(0) = u^{in} \gamma\big( v^{in} \big) \;\;\text{ in }\; \Omega. \label{d1c}
		\end{align}
	\end{subequations}
	Let $\psi\in C^1([0,\infty))$ be the solution to the ordinary differential equation 
	\begin{subequations}\label{d2}
		\begin{align}
			\frac{d\psi}{dt} & = Y  \|\gamma'(v)\|_\infty \psi \big( \|\varphi\|_\infty - \psi \big), \qquad t>0, \label{d2a}\\
			\psi(0) & = \| u^{in} \gamma(v^{in}) \|_\infty. \label{d2b}
		\end{align}
	\end{subequations}
	Owing to~\eqref{d1}, \eqref{d2}, and the parabolic comparison principle,
	\begin{equation}
		\varphi(t,x) \le \psi(t), \qquad (t,x)\in [0,\infty)\times \bar{\Omega}. \label{d3}
	\end{equation}
	In particular, for $t>0$, $\|\varphi(t)\|_\infty \le \psi(t)$ and we deduce from~\eqref{d2a} that $d\psi/dt\le 0$ on $(0,\infty)$. Combining the just established time monotonicity of $\psi$ with~\eqref{d3} gives
	\begin{equation}
		\|\varphi(t)\|_\infty \le \psi(t) \le \psi(0) = \| u^{in} \gamma(v^{in}) \|_\infty, \qquad t\ge 0. \label{d4}
	\end{equation}
	Now, on the one hand, the monotonicity of $\gamma$ and Theorem \ref{thm4} imply that
	\begin{equation}
		\gamma(\min_{\bar{\Omega}}\{v^{in}\}) u(t,x) \le \varphi(t,x) \le \|\varphi(t)\|_\infty, \qquad (t,x)\in [0,\infty) \times \bar{\Omega}. \label{d5}
	\end{equation}
	On the other hand, the elliptic comparison principle and the definition $v^{in} = \mathcal{A}^{-1}[u^{in}]$ entail that $\|v^{in}\|_\infty \le \|u^{in}\|_\infty$ and we use once more the monotonicity of $\gamma$ to obtain
	\begin{equation}
		\| u^{in} \gamma(v^{in}) \|_\infty \le \|u^{in}\|_\infty \gamma\big(\max_{\bar{\Omega}}\{v^{in}\}\big) \,. \label{d6}
	\end{equation}
	Gathering~\eqref{d4}, \eqref{d5}, and~\eqref{d6} completes the proof. 
\end{proof}

%%%%%%%%%%%%%%%%
%%%%%%%%%%%%%%%%
%%%%%%%%%%%%%%%%%%%
\section{Large time behaviour for non-decreasing $\gamma$}\label{sec5}
%%%%%%%%%%%%%%%%
%%%%%%%%%%%%%%%%

\begin{proof}[Proof of of Theorem~\ref{thm5}]
	As already pointed out in the introduction, the monotonicity of $\gamma$ implies the existence of a Liapunov function. Indeed, it readily follows from~\eqref{ks} that
	\begin{align*}
		\frac{1}{2} \frac{d}{dt} \left( \|\nabla v\|_2^2 + \|v\|_2^2 \right) & = \int_\Omega v \partial_t (v-\Delta v)\ \mathrm{d}x = \int_\Omega v \partial_t u\ \mathrm{d}x \\
		& = \int_\Omega u \gamma(v) \Delta v\ \mathrm{d}x = \int_\Omega \gamma(v) (v-\Delta v) \Delta v\ \mathrm{d}x \\
		& = - \int_\Omega (v \gamma'(v) + \gamma(v)) |\nabla v|^2\ \mathrm{d}x - \int_\Omega \gamma(v) |\Delta v|^2\ \mathrm{d}x,
	\end{align*}
and the monotonicity and positivity of $\gamma$, along with Theorem~\ref{thm4}, guarantee that the right-hand side of the above identity is non-positive. We then proceed as in the proof of \cite[Theorem~1.3]{AhYo2019} to complete that of Theorem~\ref{thm5}.
\end{proof}
 
%%%%%%%%%%%%%%%%
%%%%%%%%%%%%%%%%
\section*{Acknowledgments}
%%%%%%%%%%%%%%%%
%%%%%%%%%%%%%%%%
Jiang is supported by National Natural Science Foundation of China (NSFC)
under grants No.~12271505 \& No.~12071084,  the Training Program of Interdisciplinary Cooperation of Innovation Academy for Precision Measurement Science and Technology, CAS (No.~S21S3202), and by Knowledge Innovation Program of Wuhan-Basic Research (No.~2022010801010135).
%%%%%%%%%%%%%%%%
%%%%%%%%%%%%%%%%
\bibliographystyle{siam}
\bibliography{JJPL2023}
%%%%%%%%%%%%%%%%
%%%%%%%%%%%%%%%%

\end{document}